\documentclass[12pt, reqno]{amsart}
\usepackage{amssymb,a4}
\usepackage{amsmath,amsthm,amscd}
\usepackage{amsthm}
\usepackage{amsfonts,amssymb}
\usepackage{mathrsfs}
\usepackage{amsmath}

\numberwithin{equation}{section}
\newtheorem{theo}{Theorem}[section]
\newtheorem{defi}[theo]{Definition}

\newtheorem{lemm}[theo]{Lemma}

\newtheorem{prop}[theo]{Proposition}

\newtheorem{rema}[theo]{Remark}

\newcommand{\cc}{\mathbb{C}}

\newcommand{\z}{\mathbb{Z}}

\newcommand{\m}{\mathfrak{m}}

\newcommand{\fb}{\mathfrak{b}}

\newcommand{\mi}{\mathbf{i}}
\newcommand{\mk}{\mathbf{k}}
\newcommand{\mm}{\mathbf{m}}
\newcommand{\mj}{\mathbf{j}}

\newcommand{\p}{\mathfrak{p}}

\newcommand{\N}{\mathfrak{g}}

\newcommand{\Ind}{\mbox{Ind}}

\begin{document}

\title{Simple restricted modules for Neveu-Schwarz algebra}

\author{Dong Liu}
\address{Department of Mathematics, Huzhou University, Zhejiang Huzhou, 313000, China}
\email{liudong@zjhu.edu.cn}

\author{Yufeng Pei}
\address{Department of Mathematics, Shanghai Normal University,
Shanghai, 200234, China} \email{pei@shnu.edu.cn}

\author{Limeng Xia}

\address{Institute of Applied System Analysis, Jiangsu University, Jiangsu Zhenjiang, 212013,
China}\email{xialimeng@ujs.edu.cn}

\thanks{Mathematics Subject Classification: 17B65, 17B68, 17B70.}

\maketitle

\begin{abstract}
In this paper, we give a construction of simple
modules generalizing both highest weight and
Whittaker modules for the Neveu-Schwarz algebra, in the spirit of the work of Mazorchuk and Zhao on simple Virasoro modules.   We establish a 1-1 correspondence
between simple restricted Neveu-Schwarz modules and simple modules of a family of finite dimensional
solvable Lie superalgebras  associated to the Neveu-Schwarz algebra. Moreover,  for two of these  superalgebras all simple
modules are classified.
\end{abstract}

\section{Introduction}

It is well known that the Virasoro algebra  plays  significant roles in diverse areas of mathematics and physics.  An  important class of modules over the Virasoro algebra is the class of weight modules have been the main focus (cf.\cite{IK3} and references therein). All simple weight  modules with finite dimensional weight spaces  were classified by Mathieu in \cite{Ma} Whittaker modules  are non-weight modules
defined over finite dimensional simple Lie algebras, first appeared in \cite{AP,Kos}. These modules have been
studied subsequently in a variety of different settings(cf. \cite{BCW, BM, BO,Ch, LWZ}). Whittaker modules for the Virasoro algebra were studied in \cite{FJK,GLZ1,OW,Ya}. Recently, Mazorchuk and Zhao \cite{MZ}  proposed a uniform construction of simple Virasoro modules generalizing
and including both highest weight and various versions of Whittaker modules. They also characterize simple Virasoro modules that are locally finite over a positive part. Motivated by \cite{MZ}, new simple modules for the Virasoro algebra and its extensions have been  studied \cite{CG1,CHS,LuZ2, MW,ZC}.

Superconformal algebras have a long history in mathematical physics. The simplest examples, after the Virasoro algebra itself
(corresponding to N = 0) are the N = 1 superconformal algebras: the Neveu-Schwarz algebra  and the
Ramond algebra. These infinite-dimensional Lie superalgebras are also called the super-Virasoro algebras as they can be regarded as natural  super generalizations of the Virasoro algebra.  Weight modules for the super-Viraoro algebras  have been extensively investigated (cf. \cite{DLM,IK1,IK2,KW,MR,Su}).  The authors \cite{LPX} introduced  Whittaker type modules over  the super-Virasoro algebras and obtain necessary and sufficient conditions for  irreducibility of these modules. The aforementioned results demonstrate that the Whittaker modules satisfy some properties that their non-super analogues do. However, there are several differences and some features that are new in the super case \cite{BCW,Se}. This leads to
an additional challenge for generalizing Lie algebra results in the Lie superalgebra setting.

It is known that both simple  highest weight modules and simple Whittaker modules for  the Neveu-Schwarz algebra are restricted Neveu-Schwarz modules. In view of this,  naturally one would want to find a unified characterization for these  simple  modules. This is part of our motivation for this paper. Whittaker modules have been studied in the framework of vertex operator algebra theory in \cite{ALPY,ALZ,HY,Ta}. Irreducibility of certain weak modules for cyclic orbifold vertex algebras have been established. It is well known (cf. \cite{KW1,Li}) that  there is an isomorphism between the category of restricted Neveu-Schwarz modules and the category of weak modules for the Neveu-Schwarz vertex operator superalgebras.

In the present paper,  we focus on the classification problem for simple restricted Neveu-Schwarz modules.  We give a  construction of simple
modules generalizing and including both highest weight and
Whittaker modules for the Neveu-Schwarz algebra, in the spirit of the work of Mazorchuk and Zhao  on  simple Virasoro modules.  We establish a 1-1 correspondence
between simple restricted Neveu-Schwarz modules and simple modules of a family of finite dimensional
solvable Lie superalgebras  associated to the Neveu-Schwarz algebra.  Furthermore  we classify all simple modules for the first and the second members in this  family. The classification problem for simple modules of the other nontrivial members
is still open as far as we know. Note that any restricted Ramond module  is a weak $\sigma$-twisted module for the Neveu-Schwarz vertex operator superalgebras, where $\sigma$ is the canonical automorphism (cf. \cite{Li1}). Classification problem for simple restricted Ramond modules can be studied similarly.

This paper is organized as follows:
 In Section 2, we recall some
notations, and collect the known facts about the Neveu-Schwarz algebra.
In Section 3, we  construct simple Neveu-Schwarz modules for  generalizing and including both highest weight and
Whittaker modules. In Section 4, we provide a characterization of simple restricted modules  for Neveu-Schwarz algebra, which reduces the problem of classification of
simple restricted Neveu-Schwarz modules to classification of simple modules over a family of finite dimensional
solvable Lie superalgebras.
In Section 5, we discuss  classification of simple modules over a family of finite dimensional
solvable Lie superalgebras. We recover the Whittaker modules for the Neveu-Schwarz algebra and produce new simple Neveu-Schwarz modules.

Throughout the paper, we shall use $\cc,  {\mathbb N}$, $\z_+$ and $\z$ to denote the sets of complex numbers, non-negative integers, positive integers and integers respectively.

\section{Preliminaries} \label{sec:preliminaries}
In this section, we recall some definitions and  results for later use.

\subsection{}
Let $V = V_{\bar 0}\oplus V_{\bar 1}$ be any $\z_2$-graded vector space. Then any element $u\in V_{\bar 0}$ (${\rm resp. }$
$u\in V_{\bar 1}$) is said to be even (${\rm resp.}$ odd). We define $|u|=0$ if $u$ is even and $|u|=1$ if $u$ is odd. Elements in $V_{\bar 0}$ or $V_{\bar 1}$ are called homogeneous. Whenever $|u|$ is written, it is
understood that $u$ is homogeneous.

 Let $L=L_{\bar 0}\oplus L_{\bar 1}$ be a  Lie superalgebra, an $L$-module is a $\z_2$-graded vector space $V=V_{\bar0}\oplus V_{\bar 1}$ together with a bilinear map, $L\times V\to V$, denoted $(x,v)\mapsto xv$ such that
$$
x(yv)-(-1)^{|x||y|}y(xv)=[x,y]v
$$
for all $x,y\in L, v\in V$, and  $L_{i} V_{j}\subseteq V_{i+j}$ for all $i, j\in\z_2$. It is clear that there is a parity change functor $\Pi$ on the category of $L$-modules, which interchanges
the $\z_2$-grading of a module. We use $U(L)$ to denote the universal enveloping algebra.

\begin{defi}
Let $L$ be a  Lie superalgebra and $V$ be an $L$-module and $x\in L$. If for any $v\in V$ there exists $n\in\z_+$ such that $x^nv=0$,  then we call that the action of  $x$  on $V$ is { locally nilpotent}.
Similarly, the action of $L$   on $V$  is {locally nilpotent} if for any $v\in V$ there exists $n\in\z_+$ such that $L^nx=0$.
\end{defi}

\begin{defi}
Let $L$ be a  Lie superalgebra and $V$ be an $L$-module and $x\in L$. If for any $v\in V$ we have $\mathrm{dim}(\sum_{n\in\z_+}\cc x^nv)<+\infty$, then we call that the action of $x$ on $V$ is   locally finite.
Similarly, the action of $L$  on $V$  is locally finite
 if for any $v\in V$ we have $\mathrm{dim}(\sum_{n\in\z_+} {L}^nv)<+\infty$.
\end{defi}

\begin{rema}{\rm The action of $x$ on $V$ is locally nilpotent implies that the action of $x$ on $V$ is locally finite.
If $L$ is a finitely generated Lie superalgebra, then the action of $L$ on $V$ is locally nilpotent implies that the action of
 $L$
 on $V$ is locally finite.
}
\end{rema}

\subsection{Neveu-Schwarz algebra}

\begin{defi}\label{AD}
The Neveu-Schwarz algebra is the Lie superalgebra
$$
\N=\bigoplus_{n\in\z}\cc L_n\oplus\bigoplus_{r\in\z+\frac12} \cc G_r\oplus\cc  {\bf c}
$$
which satisfies the following commutation relations:
\begin{align*}
[L_m, L_n] &= (m-n)L_{m+n} + \delta_{m,-n} \frac{m^3-m}{12} {\bf c}, \\
[L_m, G_r] &= \left({m\over2}-r\right)G_{m+r},\\
[G_r, G_s] &= 2L_{r+s}+\frac{1}{3}\delta_{r+s,0}\left(r^2-\frac{1}{4}\right){\bf c},\\
[\N,{\bf c}] &= 0,
\end{align*}
for all $m, n\in\z$, $r, s\in\z+\frac12$, where
$$
|L_n|=\bar{0},\quad |G_r|=\bar{1},\quad |{\bf c}|=\bar{0}.
$$
\end{defi}

By definition, we have the following decompositions:
$$
\N=\N_{\bar0}\oplus\N_{\bar1},
$$
where
$$
\N_{\bar0}=\bigoplus_{n\in\z}\cc L_n\oplus\cc  {\bf c},\quad \N_{\bar1}=\bigoplus_{r\in\z+\frac12} \cc G_r.
$$
It is clear that the even part $\N_{\bar 0}$ is isomorphic to the well-known Virasoro algebra $\mathrm{Vir}$. The Neveu-Schwarz algebra $\N$ has a $\frac12\z$-grading by the eigenvalues of the adjoint action of $L_0$.
Then  $\N$ possesses the following triangular decomposition:
$$
\N=\N_{+}\oplus \N_{0}\oplus \N_{-},
$$
where
$$
\N_\pm=\bigoplus_{n\in\z_+}\cc L_{\pm n}\oplus \bigoplus_{r\in\mathbb{N}+\frac{1}2}\cc G_{\pm r},\quad \N_{0}=\cc L_0\oplus\cc {\bf c}.
$$

\begin{defi}
If $W$ is a $\N$-module on which  ${\bf c}$ acts as a complex
scalar $\ell$,  we say that $W$ is of central charge $\ell$.
\end{defi}

\begin{defi}
A $\N$-module $W$ is called  restricted in the sense that for every $w\in W$,
$$L_iw=G_{i-\frac{1}{2}}w= 0$$
for $i$ sufficiently large.
\end{defi}

Let
\begin{eqnarray*}
\fb&=&\bigoplus_{i\geq 0}\cc L_i \oplus \bigoplus_{i\geq 1}\cc G_{i-\frac12}.
\end{eqnarray*}
It is clear that $\fb$ is a subalgebra of $\N$.

Given
a $\fb$-module $V$ and $\ell\in \cc$, consider the corresponding induced module
$$
\Ind(V) :=U(\N)\otimes_{U(\fb)} V
$$
and denote
$$
\Ind_\ell(V)=\Ind(V)/({\bf c}-\ell)\Ind(V).
$$

Denote by  $\mathbb{M}$ the set of all infinite vectors of the form $\mi:=(\ldots, i_2, i_1)$ with entries in $\mathbb N$,
satisfying the condition that the number of nonzero entries is finite, and $\mathbb{M}_1:=\{\mathbf{i}\in \mathbb{M}\mid i_k=0, 1, \ \forall k\in\z_+\}$.

Let $\mathbf{0}$ denote the element $(\ldots, 0, 0)\in\mathbb{M}$ and
for
$i\in\z_+$ let $\epsilon_i$ denote the element $(\ldots,0,1,0,\ldots,0)\in\mathbb{M}$,
where $1$ is
in the $i$'th  position from right. For any $\mi\in\mathbb{M}$, we write
$${\rm w}(\mi)=\sum_{k\in\z_+}k\cdot i_k,$$
which is a nonnegative integer. For any nonzero  $\mi\in\mathbb{M}$, let $p$ be the smallest integer such that $i_p \neq0$ and define  $\mi^\prime=\mi-\epsilon_p$.

\newpage

\begin{defi}\rm
 Denote by $\prec$ the  {\em reverse  lexicographical  total order}  on  $\mathbb{M}$,  defined as follows: for any $\mi,\mk\in\mathbb{M}$, set $$\mi\prec \mathbf{j}\ \Leftrightarrow \ \mathrm{ there\ exists} \ r\in\z_+ \ \mathrm{such \ that} \ i_r<j_r\ \mathrm{and} \ i_s=j_s,\ \forall\; 1\le s<r.$$
Now we can induce a {\em principal total order} on $\mathbb{M}\times\mathbb{M}_1$, still denoted by $\prec$:
$(\mi,\mk) \prec (\mathbf{j},\mathbf{m})$ if and only if one of the following condition is satisfied:

(1) $\rm{w}(\mi)+\rm{w}(\mk)<\rm{w}(\mj)+\rm{w}(\mathbf{m})$;

(2) $\rm{w}(\mi)+\rm{w}(\mk)=\rm{w}(\mj)+\rm{w}(\mathbf{m})$ and $\mk \prec \mathbf{m}$;

(3) $\rm{w}(\mi)+\rm{w}(\mk)=\rm{w}(\mj)+\rm{w}(\mathbf{m})$ and $\ \mk=\mathbf{m} \
\mathrm{and }\ \mi \prec \mathbf{j},\quad \forall\mi,\mathbf{j}\in\mathbb{M}, \mk, \mathbf{m}\in\mathbb{M}_1$.
\end{defi}

Let $V$ be a simple $\fb$-module. For $\mk\in\mathbb{M}_1, \mi\in\mathbb{M}$, we denote
$$G^{\mk} L^{\mi}=\ldots G_{-2+\frac{1}{2}}^{k_2} G_{-1+\frac{1}{2}}^{k_1}\ldots L_{-2}^{k_2} L_{-1}^{k_1}\in U(\N_{-}).$$
According to the $\mathrm{PBW}$ Theorem, every element of $\mathrm{Ind}_\ell(V)$ can be uniquely written in the
following form
\begin{equation}\label{def2.1}
\sum_{\mk\in\mathbb{M}_1,\mi\in\mathbb{M}}G^{\mk} L^{\mi} v_{\mk,\mi},
\end{equation}
where all  $v_{\mk,\mi}\in V$ and only finitely many of them are nonzero. For any $v\in\mathrm{Ind}_\ell(V)$ as in  \eqref{def2.1}, we denote by $\mathrm{supp}(v)$ the set of all $(\mi,\mk)\in \mathbb{M}\times \mathbb{M}_1$  such that $v_{\mk, \mi}\neq0$.
 For a nonzero $v\in \mathrm{Ind}_\ell(V)$, we write $\mathrm{deg}(v)$  the maximal (with respect to the principal total order on $\mathbb{M}\times\mathbb{M}_1$) element in $\mathrm{supp}(v)$, called the { degree} of $v$. Note that here and later
 we make the
convention that  $\mathrm{deg}(v)$ only for  $v\neq0$.

\section{Construction of simple restricted $\N$-modules}

In this section, we give a construction of simple restricted $\N$-modules.

\begin{theo}\label{simple-theo}
Let $V$ be a simple $\fb$-module and  assume that there exists  $t\in\z_+$ satisfying the following two conditions:
\begin{itemize}
\item[(a)] the action of  $L_{t}$   on  $V$ is injective;
\item[(b)] $L_iV=0$ for all $i>t$.
\end{itemize}
Then
\begin{itemize}
\item[(i)] $G_{j-\frac12}V=0$ for all $j>t$.
\item[(ii)] For $\ell\in\cc$, the induced module $\mathrm{Ind}_\ell(V)$ is a simple $\N$-module.
\end{itemize}
\end{theo}

\begin{proof}
(i)\ Assume $L_iV=0$ for all $i>t$.  For $j\geq t$, by $G_{j+\frac12}^2V=L_{2j+1}V=0$, we have $W=G_{j+\frac12}V$ is a proper subspace of $V$.   For $r\in\mathbb{Z}_+$, we have
$$
G_{r-\frac{1}{2}}W=G_{r-\frac{1}{2}}G_{j+\frac12}V=L_{r+j}V-G_{j+\frac12}G_{r-\frac{1}{2}}V\subset G_{j+\frac12}V=W,
$$
$$
2L_rW=[G_{r-\frac{1}{2}},G_{\frac{1}{2}}]G_{j+\frac12}V
=G_{r-\frac{1}{2}}G_{\frac{1}{2}}G_{j+\frac12}V+G_{\frac{1}{2}}G_{r-\frac{1}{2}}G_{j+\frac12}V\subset W
$$
It follows that $W$ is a proper submodule of $V$. Then $W=G_{j+\frac12}V=0$ for $j\geq t$ since $V$ is simple.

In order to prove (ii), we need the following claim.

{\bf Claim.}\ For any $v\in\ \mathrm{Ind}_\ell(V)\setminus V$, let $\mathrm{deg}(v)=(\mi,\mk)$, ${\hat{k}}=\mathrm{min}\{s:k_s\neq0\}$ if $\mk\neq\mathbf{0}$ and  $\hat{i}=\mathrm{min}\{s:i_s\neq0\}$ if $\mi\neq\mathbf{0}$.  Then

\begin{itemize}
\item[(1)] If $\mk\neq\mathbf{0}$, then $\hat{k}>0$ and $\mathrm{deg}(G_{{\hat{k}}+t-\frac{1}{2}}v)=(\mi,\mk^{\prime})$;
\item[(2)] If $\mk=\mathbf{0}, \mi\neq\mathbf{0}$,  then $\hat{i}>0$ and $\mathrm{deg}(L_{\hat{i}+t}v)=(\mi^{\prime}, \mathbf{0})$.
\end{itemize}

To prove this, we assume that
\begin{equation}
v=\sum_{\mm\in\mathbb{M}_1,\,\mj\in\mathbb{M}}G^{\mm} L^{\mj} v_{\mm,\mj},
\end{equation}
where all  $v_{\mm, \mj}\in V$ and only finitely many of them are nonzero.

{\rm (1)} It suffices to consider those $v_{\mathbf{j},\mathbf{m}}$
with $$G^{\mathbf{m}}L^{\mathbf{j}}v_{\mathbf{j},\mathbf{m}}\neq0.$$
Note that $G_{\hat{k}+t-\frac12}v_{\mathbf{j},\mathbf{m}}=0$ for any  $(\mathbf{j},\mathbf{m})\in \mathrm{supp}(v)$.
One can easily check that
$$G_{\hat{k}+t-\frac12}G^{\mathbf{m}}L^{\mathbf{j}}v_{\mathbf{j},\mathbf{m}}=[G_{\hat{k}+t-\frac12}, G^{\mathbf{m}}]L^{\mathbf{j}}v_{\mathbf{j},\mathbf{m}}\pm G^{\mathbf{m}}[G_{\hat{k}+t-\frac12},L^{\mathbf{j}}]v_{\mathbf{j},\mathbf{m}}.$$
Clearly $L_{t}v_{\mathbf{j},\mathbf{m}}\neq0$ by (a).

If
$$
{\rm w}(\mj)+{\rm w}(\mathbf{m})<{\rm w}(\mi)+{\rm w}(\mathbf{k}),
$$
then
$$
\deg G_{\hat{k}+t-\frac12}G^{\mathbf{m}}L^{\mathbf{j}}v_{\mathbf{j},\mathbf{m}}\prec (\mi, \mk').
$$

Now we suppose that ${\rm w}(\mj)+{\rm w}(\mathbf{m})={\rm w}(\mi)+{\rm w}(\mk)$ and $ (\mathbf{m},\rm{w}(\mathbf{m})) \prec(\mk,\rm{w}(\mk))$ and denote  $$\mathrm{deg}(G_{\hat{k}+t-\frac12}G^{\mathbf{m}}L^{\mathbf{j}}v_{\mathbf{j},\mathbf{m}})
=(\mathbf{j}_1,\mathbf{m}_1)\in \mathbb{M}\times \mathbb{M}_1.$$

Let $\hat{m}:=\mathrm{min}\{s:m_s\neq0\}>0$.
If $\hat{m}>\hat{k}$, then ${\rm w}(\mj_1)+{\rm w}(\mathbf{m}_1)<{\rm w}(\mi)+{\rm w}(\mathbf{k}')$ and then
$\mathrm{deg}(G_{\hat{k}+t-\frac12}G^{\mathbf{m}}L^{\mathbf{j}}v_{\mathbf{j},\mathbf{m}})
=(\mathbf{j}_1,\mathbf{m}_1)\prec(\mi,\mk^{\prime})$.
If $\hat{m}=\hat{k}$,
then $(\mathbf{j}_1,\mathbf{m}_1)=(\mathbf{i}, \mathbf{m}^{\prime})$.
Since $ \mathbf{m}^{\prime} \prec\mk^{\prime}$, we also have $(\mathbf{j}_1,\mathbf{m}_1)\prec(\mi,\mk^{\prime})$.

If ${\rm w}(\mj)+{\rm w}(\mathbf{m})={\rm w}(\mi)+{\rm w}(\mk)$ and $\mathbf{m}=\mk$, it is easy to see that
$$\mathrm{deg}([G_{\hat{k}+t}, G^{\mathbf{m}}]L^{\mathbf{j}}v_{\mathbf{j},\mathbf{m}})=
(\mathbf{j},\mk^{\prime})\preceq (\mi,\mk^{\prime}),$$
$$\mathrm{deg}(G^{\mathbf{m}}[G_{\hat{k}+t}, L^{\mathbf{j}}]v_{\mathbf{j},\mathbf{m}})=
(\mathbf{j}_1, \mathbf{m}_1)\prec (\mi, \mk'),$$
where the equality holds if and only if $\mathbf{j}=\mi$.

Combining all the arguments above we conclude that
$\mathrm{deg}(G_{\hat{k}+t-\frac12}v)=(\mi, \mk^{\prime})$, as desired.

{\rm (2)}  We consider   $v_{\mathbf{j},\mathbf{0}}$
with $$L_{{\hat{i}}+t}L^{\mathbf{j}}v_{\mathbf{j},\mathbf{0}}\neq0.$$
Since $L_{{\hat{i}}+t}v_{\mathbf{j},\mathbf{0}}=0$ for any  $(\mathbf{j},\mathbf{0})\in \mathrm{supp}(v)$,
then we have
$$L_{{\hat{i}}+t}L^{\mathbf{j}}v_{\mathbf{j},\mathbf{0}}=
[L_{{\hat{i}}+t}, L^{\mathbf{j}}]v_{\mathbf{j},\mathbf{0}}.$$

If $\mathbf{j}=\mi$, it is easy to get that
$$\mathrm{deg}(L_{{\hat{i}}+t}L^{\mathbf{j}}v_{\mathbf{j},\mathbf{0}})=
(\mathbf{j}^{\prime},\mathbf{0})=(\mi^{\prime},\mathbf{0}).$$

Now suppose that
\begin{equation}\mathrm{deg}(L_{{\hat{i}}+t}L^{\mathbf{j}}v_{\mathbf{j},\mathbf{0}})
=(\mathbf{j}_1, \mathbf{0}).\label{eq333}\end{equation}

If ${\rm w}(\mathbf{j})<{\rm w}(\mathbf{i})$, then ${\rm w}(\mathbf{j_1})\le {\rm w}(\mathbf{j})-\hat{i}<{\rm w}(\mathbf{i})-\hat{i}={\rm w}(\mathbf{i'})$, which shows that $(\mathbf{j}_1, \mathbf{0})\prec (\mathbf{i}, \mathbf{0})$.

If ${\rm w}(\mathbf{j})={\rm w}(\mi)$ and ${\mathbf{j}}\prec {\mi}$. Let $\hat{j}:=\mathrm{min}\{s:j_s\neq0\}>0$.
If $\hat{j}>{\hat{i}}$, we obtain ${\rm w}(\mathbf{j}^\prime)<{\rm w}(\mathbf{i})-{\hat{i}}={\rm w}(\mi^{\prime})$. If $\hat{j}={\hat{i}}$,
we can similarly check that $(\mathbf{j}_1, \mathbf{0})=(\mathbf{j}^{\prime}, \mathbf{0})$.
By $\mathbf{j}^{\prime}\prec \mi^{\prime}$, we have $\mathrm{deg}(L_{{\hat{j}}+t}L^{\mathbf{j}}v_{\mathbf{j},\mathbf{0}})
=(\mathbf{j}_1, \mathbf{0})\prec(\mi^{\prime},\mathbf{0})$.

Consequently, we conclude that
$\mathrm{deg}(L_{{\hat{j}}+t}v)=(\mi^{\prime},\mathbf{0})$.  This proves the claim.

Using the claim repeatedly, from any nonzero element $v\in\mathrm{Ind}_\ell(V)$ we can reach a nonzero element in
 $U(\N)v\cap V\neq0$, which implies that the simplicity of $\mathrm{Ind}_\ell(V)$.
\end{proof}

\begin{rema}{\rm In Theorem \ref{simple-theo}, note that the actions of  $L_i,G_{i-\frac12}$ on $\mathrm{Ind}_\ell(V)$  for  all $i>t$  are locally nilpotent. It follows that $\mathrm{Ind}_\ell(V)$ is a simple restricted  $\N$-module of central charge $\ell$.
}
\end{rema}

\section{Characterization of simple restricted $\N$-modules}

In this section, we give a characterization of simple restricted $\N$-modules of central charge $\ell$.

For  $t\in \mathbb{N}$, let
\begin{eqnarray*}
\m^{(t)} &=& \bigoplus_{m> t}\cc L_m \oplus \bigoplus_{m> t}\cc G_{m-\frac12},
\end{eqnarray*}
Note that $\m^{(0)}=\N_+$.

\begin{prop}\label{th2}
Let $S$ be a simple $\N$-module. Then the following conditions are equivalent:

\begin{itemize}
  \item[(1)]   There exists $t\in\z_+$ such that the actions of $L_i,G_{i-\frac{1}{2}}$ for all $i\ge t$ on $S$ are locally finite.
 \item[(2)]   There exists $t\in\z_+$ such that the actions of $G_{i-\frac{1}{2}}, L_i$ for all $i\ge t$ on $S$ are locally nilpotent.
 \item[(3)]  There exist $t\in\z_+$ such that  $S$ is a locally finite $\m^{(t)}$-module.
 \item[(4)]  There exist $t\in\z_+$ such that $S$ is a locally nilpotent $\m^{(t)}$-module.
 \item[(5)] $S$ is a highest weight module, or there exists $\ell\in \cc$, $t\in\z_+$ and  a simple $\fb$-module $V$ such that both conditions $(a)$ and $(b)$ in Theorem \ref{simple-theo} are satisfied  and
$S\cong\mathrm{Ind}_\ell(V)$.
\end{itemize}

\end{prop}
\begin{proof}
First we prove $(1)\Rightarrow(5)$. Suppose that $S$ is a simple $\N$-module and there exists $t\in\z_+$ such that the actions of
$G_{i-\frac{1}{2}}, L_i, i>t$ are locally finite.

Choose a simple $\mathrm{Vir}$-submodule $S'$ of $S$. Clearly $L_i, i\ge t$ are locally finite on $S'$. By Proposition 4 in \cite{MZ}, there exist $t'\in\z_+\, (t'\ge t)$  and a simple ${\mathrm{Vir}_+}$-module $W$ such that $S'={\rm Ind}(W)$ as Virasoro module and $L_nW=0$ for all $n\ge t'$, where ${\mathrm{Vir}_+}:=\bigoplus_{m>0}\cc L_m$.

Then we can choose a nonzero $w\in W$ such that $L_nw=0$ for all $n\ge t'$.

Take any $j\in\z$ with $j>t'$ and we denote
 \begin{eqnarray*}
 && V_G=\sum_{m\in\mathbb N}\cc L_{t'}^mG_{j-\frac{1}{2}}w=U(\cc L_{t'})G_{j-\frac{1}{2}}w,
 \end{eqnarray*}
 which are all finite-dimensional. By Definition \ref{AD}, it is clear that $G_{j+(m+1)t'-\frac{1}{2}}w\in V_G$ if $G_{j+mt'-\frac{1}{2}}w\in V_G$.
By induction on $m$, we obtain $G_{j+mt'-\frac{1}{2}}w\in V_G$ for all $m\in \mathbb N$.
So $\sum_{m\in\mathbb N}\cc G_{j+mt'-\frac{1}{2}}w$ are
finite-dimensional for any $j>t'$, and then
 \begin{eqnarray*}
\sum_{i\in\z_+}\cc G_{t'+i-\frac{1}{2}}w=\cc G_{t'+\frac{1}{2}}w+\sum_{j=t'+1}^{2t'}\Big(\sum_{m\in\mathbb N}\cc G_{j+mt'-\frac{1}{2}}w\Big)
\end{eqnarray*}
 is finite-dimensional.

Now we can take $p\in\z_+$ such that
\begin{equation}
\sum_{i\in\z_+}\cc G_{t'+i-\frac{1}{2}}w\!=\!\sum_{i=0}^{p}\cc  G_{t'+i+\frac{1}{2}}w.
\end{equation}
Set $V^\prime:=\sum_{\imath_0,\ldots,\imath_p\in \{0, 1\}}\cc G_{t'+\frac{1}{2}}^{\imath_0}\cdots G_{t'+p+\frac{1}{2}}^{\imath_k}w$, which is finite-dimensional by (1). Moreover $V^\prime$ is a finite-dimensional $\m^{(t')}$-module.

It follows that we can  choose
 a minimal $n\in\mathbb N$ such that
 \begin{equation}\label{lm3.3}
 (a_0G_{m-\frac12}+a_1G_{m+1-\frac12}+\cdots + a_{n}G_{m+{n}-\frac12})V^\prime=0
 \end{equation}
 for some $m> t'$ and  $a_i\in \cc$.
 Applying $L_{2m-1}$ to \eqref{lm3.3}, one has
$$(a_1[L_{2m-1},G_{m+\frac12}]+\cdots +a_{n}[L_{2m-1},G_{m+{n}-\frac12}])V^\prime=0$$ since $L_{2m-1}V^\prime\subset V^\prime$.
Then $$(a_1G_{3m-\frac12}+\cdots +na_{n}G_{3m+{n}-1})V^\prime=0,$$
which implies $n=0$, that is, \begin{equation}G_{m-\frac12}V^\prime=0. \label{eq111}\end{equation}
By action of $L_i$ on (\ref{eq111})\begin{equation}G_{m+i-\frac{1}{2}}V^\prime=0, \ \forall i>t'.\end{equation}

 For any $\tilde{k}\in \mathbb N$, we consider the vector space
 $$N_{\tilde{k}}=\{v\in S\mid G_{k-\frac{1}{2}}v=L_kv=0 \quad \mbox{for\ all}\ k>\tilde{k}\}.$$
Clearly, $N_{{\tilde{k}}}\neq0$ for sufficiently large $\tilde{k}\in\mathbb N$.
Thus we can find a smallest nonnegative integer, saying $s$, with $V:=N_{s}\neq 0$.
Using $k>s$ and $p\geq 1$, it follows from $k+p-\frac{1}{2}> s$ and $k+p-1> s$  that we can easily check that
$$L_{k}(G_{p-\frac{1}{2}}v)=(p-\frac{k+1}{2})G_{k+p-\frac{1}{2}}v=0$$
and
$$
G_{k-\frac{1}{2}}(G_{p-\frac{1}{2}}v)=2L_{k+p-1}v=0,
$$
for any $v\in V$, respectively.
Clearly, $G_{p-\frac{1}{2}}v\in V$ for
all $p\geq 1$. Similarly, we can also obtain $L_iv\in V$
for all $i\in \mathbb N$. Therefore, $V$ is a $\fb$-module.

If $s=0$, then by Theorem 1\,(c) in \cite{MZ2} and the proofs in Theorem \ref{simple-theo} (i) , $S$ is a highest weight module.

If $s\ge1$, by the definition of $V$, we can obtain that the action of  $L_{s}$  on $V$ is injective by Theorem \ref{simple-theo}. Since $S$ is simple
and generated by $V$, then there exists a canonical surjective map
$$\pi:\mathrm{ Ind}(V) \rightarrow S, \quad \pi(1\otimes v)=v,\quad \forall  v\in V.$$
Next we only need to show that $\pi$ is also injective, that is to say, $\pi$ as the canonical map is bijective.  Let $K=\mathrm{ker}(\pi)$. Obviously, $K\cap V=0$. If $K\neq0$, we can choose a nonzero vector $v\in K\setminus V$ such that $\mathrm{deg}(v)=(\mi,\mk)$ is minimal possible.
Note that $K$ is a $\N$-submodule of $\mathrm{Ind}_\ell(V)$.
By the claim in proof of Theorem \ref{simple-theo} we can create a new vector $u\in K$  with $\mathrm{deg}(u)\prec(\mi,\mk)$, which is a contradiction. This forces $K=0$, that is, $S\cong \mathrm{Ind}_\ell(V)$. According to  the property of induced modules, we see that $V$ is simple as a $\fb$-module.

Moreover, $(5)\Rightarrow(3)\Rightarrow(1)$, $(5)\Rightarrow(4)\Rightarrow(2)$
 and $(2)\Rightarrow(1)$ are clear. This completes the proof.
\end{proof}

\begin{lemm}\label{RL}
Let $V$ be a simple restricted $\N$-module. Then there exists $t\in\z_+$ such that the actions of $L_i, G_{i-\frac{1}{2}}$ for all $i\ge t$ on $V$ are locally nilpotent.
\end{lemm}
\begin{proof}Let $0\neq v\in V$, there exists $s\in \mathbb{Z}_+$ such that $L_iv=G_{i-\frac{1}{2}}v=0$ for all $i\geq s$.
For $V=U(\mathfrak{g})v$, every element $w$ of $V$ can be uniquely written in the
following form
$$
w=\sum_{\mk\in\mathbb{M}_1,\mi\in\mathbb{M}}G^{\mk} L^{\mi}v.
$$
Then, for $i\geq s$, there exists  $N$ sufficiently large such that
$$
L_i^Nw=G_{i-\frac{1}{2}}^Nw=0.
$$

\end{proof}

From  Theorem \ref{th2} and Lemma \ref{RL}, we are in a position to state our main result.

\begin{theo}\label{MT}
  Let $\ell\in\cc$, every simple restricted $\N$-module of central charge $\ell$ is isomorphic to
a  simple highest weight module, or  a simple module of the form $\mathrm{Ind}_\ell(V)$, where $V$ is a simple $\fb^{(t)}$-module for some $t\in\mathbb{Z}_+$, where
   $\fb^{(t)}=\fb/\m^{(t)}$ is the quotient algebra of $\fb$ by $\m^{(t)}$.
\end{theo}

\section{Simple $\fb^{(t)}$-modules and examples}

\subsection{Classifications}

In this section we mainly consider simple modules over the quotient algebra $\fb^{(t)}$ for $t\in\mathbb N$.
 For $t=0$, the algebra $\fb^{(0)}$ is commutative and its simple modules are
one dimensional. Next we shall classify all simple $\fb^{(t)}$-module for $t=1,2$.

For $t=1$, $\fb^{(1)}=\fb^{(1)}_{\bar 0}\oplus \fb^{(1)}_{\bar 1}$ is a 3-dimensional sovable Lie superalgebra with $\fb^{(1)}_{\bar 0}=\cc L_0\oplus \cc L_1$ and $\fb^{(1)}_{\bar 1}=\cc G_{\frac{1}{2}}$. All simple $\fb^{(1)}_{\bar 0}$-modules are constructed by  R. Block in \cite{Bl}.

\begin{prop}\label{pp51}
Any simple $\fb^{(1)}$-module $V$ is isomorphic to $U\oplus G_{\frac12}U$, where $U$ is a simple $\fb^{(1)}_{\bar 0}$-module, up to parity-change. In particular, $V$ is  one-dimensional if $G_{\frac12}U=0$.
\end{prop}
\begin{proof}
Let $V$ be a simple $\fb^{(1)}$-module. If $V_{\bar1}=0$, then $G_{\frac12}V=0=L_1V$, then $V$ is a simple $\fb^{(0)}$-module. In this case $\dim V=1$.

So we can suppose that $V_{\bar0}\ne 0$ and $V_{\bar1}\ne 0$.
Choose a simple $\frak b_{\bar{0}}^{(1)}$-submodule $U$ of $V_{\bar{0}}$, then $U+G_{\frac12}U$ becomes a $\fb^{(1)}$-submodule of $V$, and then $V=U+G_{\frac12}U$.

%
%
%
%
\end{proof}

For $t=2$, $\fb^{(2)}=\fb^{(2)}_{\bar 0}\oplus \fb^{(2)}_{\bar 1}$ is a 5-dimensional solvable Lie superalgebra with $\fb^{(1)}_{\bar 0}=\cc L_0\oplus \cc L_1\oplus \cc L_2$ and $\fb^{(1)}_{\bar 1}=\cc G_{\frac{1}{2}}\oplus \cc G_{\frac32}$. All simple $\fb^{(2)}_{\bar 0}$-modules are constructed by Mazorchuk and Zhao in \cite{MZ}.

\begin{prop}
Any simple module over $\fb^{(2)}$ is isomorphic to $U\oplus G_{\frac12}U$, where $U$ is a simple $\fb^{(2)}_{\bar 0}$-module, up to parity-change.
\end{prop}
\begin{proof}
Let $V=V_{\bar0}\oplus V_{\bar1}$ be a simple $\fb^{(2)}$-module.  As in Proposition \ref{pp51}, we can suppose that $V_{\bar0}\ne 0$ and $V_{\bar1}\ne 0$.

If $G_{\frac32}V=0$, then $L_2V=\frac12(G_{\frac12}G_{\frac32}+G_{\frac32}G_{\frac12})V=0$. In this case $V$ becomes a simple $\fb^{(1)}$-module.
The proposition follows from Proposition \ref{pp51}.

Now assume $G_{\frac32}V\not=0$. Due to that $G_{\frac32}G_{\frac32}V=0$,  choose a simple $\fb_{\bar{0}}^{(2)}$-submodule $U$ of $G_{\frac32}V_{\bar{1}}$ (or $G_{\frac32}V_{\bar{0}}$), then $G_{\frac32}U=0$ and $U\oplus G_\frac12U$ is a $\fb^{(2)}$-submodule.
So $V=U\oplus G_\frac12U$.
\end{proof}

\subsection{Examples}

\subsubsection{Highest weight modules}
For $h, \ell\in\cc$, let $\cc v$  be one-dimensional  ${\N}_{0}$-module  defined by
$ L_0v=h v, {\bf c}  v=\ell v.$
Let ${\N}_{+}$ act trivially on $v$, making $v $ a $({\N}_{0}\oplus{\N}_{+})$-module.
The {Verma module} for Neveu-Schwarz algebra (cf. \cite{KW1}) can be defined by
$$ M(h, \ell)=U(\N)\otimes_{U(\N_{0}\oplus\N_{+})}\cc v,$$
The module $M(h, \ell)$ has the unique simple quoitent $L(h, \ell)$, the unique (up to isomorphism) simple
highest weight module with highest weight $(h, \ell)$.  These simple modules correspond to the highest weight module in Theorem \ref{MT}.

\subsubsection{Whittaker modules}

Let
$$\mathfrak{p} = \bigoplus_{m>0}\cc L_m \oplus \bigoplus_{m>0}\cc G_{m+\frac12}.$$
and $\psi$ a Lie superalgebra homomorphism $\psi:\mathfrak{p}\to\cc$. It follows that $\psi(L_i) = 0$ for $i \ge 3$ and $\psi(G_{j-\frac12})=0$ for all $j\ge 2$.  For $\ell\in\cc$, let $\cc w$ be one dimensional $(\mathfrak{p}\oplus\cc {\bf c})$-module with $xw=\psi(x)w$ for all $x\in\mathfrak{p}$ and ${\bf c}w=\ell w$, then the Whittaker module for  Neveu-Schwarz algebra is defined by
$$
W({\psi,\ell})= U(\N) \otimes_{U(\mathfrak{p}\oplus\cc {\bf c})}\cc w.
$$
By  \cite{LPX}, the  Whittaker module $W({\psi,\ell})$ is simple if $\psi$ is non-trivial, i.e., $\psi(L_1)\ne0$ or $\psi(L_2)\ne0$.

Let $\psi:\mathfrak{p}\to\cc$ be  a nontrivial Lie superalgebra homomorphism and
 $A_{\psi}=\cc w\oplus \cc u$ with
 be a two-dimensional vector space with
\begin{eqnarray*}
x w &=&\psi(x)w, \
G_{\frac{1}{2}} w=u,\quad \forall x\in\mathfrak{p}.
\end{eqnarray*}
Then $A_{\psi}$ is a simple $\N_+$-module. Consider induced module
$$
V_{\psi}=U(\fb)\otimes_{U(\mathfrak{p})}A_{\psi}\cong \cc[L_0]A_{\psi}.
$$
It is straightforward to check that $V_{\psi}$ is a simple $\fb$-module. Hence, by Theorem \ref{simple-theo}, we obtain the corresponding simple induced  $\N$-module ${\rm Ind}_\ell(V_{\psi})$. These are exactly the Whittaker modules $W({\psi,\ell})$.

\subsubsection{High order Whittaker modules}

For  $t\in \mathbb{N}$, let
\begin{eqnarray*}
\mathfrak{p}^{(t)} &=& \bigoplus_{m> t}\cc L_m \oplus \bigoplus_{m> t}\cc G_{m+\frac12}.
\end{eqnarray*}
It is clear that $\p^{(0)}=\p$. All finite dimensional simple modules over $\p^{(0)}$ have been classified in \cite{LPX}.  Now we shall classify  all finite-dimensional simple modules over   ${\p}^{(t)}$ for $t\in\mathbb{Z}_+$.

We have the following lemma which can be proved in a way similar to Proposition 3.3 in \cite{LPX}, we have

\begin{lemm}\label{prop15}
Let $V$ be a finite dimensional simple  ${\p}^{(t)}$-module for $t\in \mathbb{Z}_+$. Then
there exists $k\in\mathbb{Z}_+$ such that  ${\p}^{(k)}V=0$.
\end{lemm}

\begin{prop}\label{cor16}
Let $S$ be a simple finite dimensional ${\p}^{(t)}$-module for $t\in\mathbb{Z}_+$. Then

\rm{(i)}  $\dim S=1$;

\rm{(ii)} $[{\p}_{\bar1}^{(t)},{\p}_{\bar1}^{(t)}]S=0$.
\end{prop}

\begin{proof}

By Proposition~\ref{prop15}, we have ${\p}^{(i)}S=0$ for some
$i\ge t$. Hence $S$ is a simple finite dimensional module
over the nilpotent Lie superalgebra  ${\p}^{(t)}/{\p}^{(i)}$. Moreover $[{\p}^{(t)}_{\bar1}, {\p}^{(t)}_{\bar1}]\subset [{\p}^{(t)}_{\bar0}, {\p}^{(t)}_{\bar0}]$, then  (i) follows from Lemma 1.37 in \cite{CW}). As dim\,$S=1$, we also have that the Lie superalgebra
$\mathrm{End}_{\mathbb{C}}(S)$ is commutative, which implies (ii). This completes the proof.
\end{proof}

Let $\psi_k$ be a Lie superalgebra homomorphism $\psi_k:{\p}^{(k)}\to \cc$ for some $k\in\mathbb{Z}_+$. Then
$\psi_k(L_i) = 0$ for $i \ge 2k+3$ and $\psi_k(G_{j+\frac12})=0$ for all $j\ge k+1$. Let $\cc w$ be one-dimensional ${\p}^{(k)}$-module with $xw=\psi_k(x)w$ for all $x\in{\p}^{(k)}$ and ${\bf c}w=\ell w$ for some $\ell\in\cc$. The higher order Whittaker module $W(\psi_k,\ell)$ is given by
$$
W(\psi_k,\ell)= U(\N) \otimes_{U({\p}^{(k)})}\cc w.
$$
\begin{prop}
For $k\in\mathbb{Z}_+$, the higher order Whittaker module $W(\psi_k,\ell)$ is simple if and only if $\psi_k(L_{2k+1})\ne 0$ or $\psi_k(L_{2k+2})\ne 0$.
\end{prop}
\begin{proof}
For $k\in\mathbb{Z}_+$,   let $\psi_k:\mathfrak{p}^{(k)}\to\cc$ be  a  Lie superalgebra homomorphism and
 $A_{\psi_k}=\cc w\oplus \cc u$ with
 be a two-dimensional vector space with
\begin{eqnarray*}
x w &=&\psi_k(x)w, \
G_{k+\frac{1}{2}} w=u,\quad \forall x\in\mathfrak{p}^{(k)}.
\end{eqnarray*}
Then $A_{\psi_k}$ is a simple $\m^{(k)}$-module if and only if  $\psi_k(L_{2k+1})\ne 0$ or $\psi_k(L_{2k+2})\ne 0$.  Moreover, if
$\psi_k(L_{2k+1})=\psi_k(L_{2k+2})= 0$,  then $A_{\psi_k}$ has a trivial submodule $\cc G_{k+\frac12}w$.

Let
$$
V_{\psi_k}=U(\fb)\otimes_{U(\mathfrak{m}^{(k)})}A_{\psi_k}.
$$
It is straightforward to check that $V_{\psi_k}$ is a simple $\fb$-module if and only if  $\psi_k(L_{2k+1})\ne 0$ or $\psi_k(L_{2k+2})\ne 0$. From Theorem \ref{simple-theo}, we obtain the corresponding induced  $\N$-module ${\rm Ind}_\ell(V_{\psi_k})$ is simple if and only if $\psi_k(L_{2k+1})\ne 0$ or $\psi_k(L_{2k+2})\ne 0$. These modules are exactly the higher order Whittaker modules $W({\psi_k,\ell})$.
\end{proof}

 \centerline{\bf ACKNOWLEDGMENTS}

\vskip5pt We gratefully acknowledge the partial financial support from the NNSF (No.11871249, No.11771142), the ZJNSF(No.LZ14A010001), the Shanghai Natural Science Foundation (No.16ZR1425000) and the Jiangsu Natural Science Foundation(No.BK20171294). We would like to thank Yanan Cai and Rencai Lu for useful discussions.

\end{document}